\documentclass[graybox, envcountsame, envcountsect]{svmult}

\usepackage{type1cm}        
\usepackage{makeidx}         
\usepackage{graphicx}        
\usepackage{multicol}        
\usepackage[bottom]{footmisc}

\usepackage{newtxtext}       %
\usepackage[varvw]{newtxmath}       

\makeindex             

\begin{document}

\title*{Tight Heffter arrays from finite fields}
\author{Marco Buratti}
\institute{Marco Buratti \at Dipartimento di Scienze di Base e Applicate per l'Ingegneria (S.B.A.I.), 
Universit\`a Sapienza di Roma, Italy, \email{marco.buratti@uniroma1.it}}
%
%
\maketitle

\begin{center}{\sl Dedicated to Doug Stinson on the occasion of his 66th birthday.}\end{center}

\abstract*{After extending in the obvious way
the classic notion of a tight Heffter array H$(m,n)$ to any group of order $2mn+1$,
we give direct constructions for elementary abelian tight Heffter arrays, hence in
particular for prime tight Heffter arrays (whose existence was already known).
If $q=2mn+1$ is a prime power,
we say that an elementary abelian H$(m,n)$ is ``over $\mathbb{F}_q$" since, for its construction, 
we exploit both the additive and multiplicative structure of the field of order $q$. 
We show that in many cases a direct construction of an H$(m,n)$ over $\mathbb{F}_q$, say $A$,  
can be obtained very easily by imposing that $A$ has rank 1 and, possibly, a rich group of {\it multipliers},
that are elements $u$ of $\F_q$ such that $u A=A$ up to 
a permutation of rows and columns. An H$(m,n)$ over $\mathbb{F}_q$ will be
said {\it optimal} if the order of its group of multipliers is the least common
multiple of the odd parts of $m$ and $n$, since this is the maximum possible
order for it.
The main result is an explicit construction of a rank-one H$(m,n)$ --
reaching almost always the optimality --
for all admissible pairs $(m,n)$ for which there exist two distinct
odd primes $p$, $p'$ dividing $m$ and $n$, respectively.}

\abstract{After extending the classic notion of a tight Heffter array H$(m,n)$ to any group of order $2mn+1$,
we give direct constructions for elementary abelian tight Heffter arrays, hence in
particular for prime tight Heffter arrays.
If $q=2mn+1$ is a prime power,
we say that an elementary abelian H$(m,n)$ is ``over $\mathbb{F}_q$" since, for its construction, 
we exploit both the additive and multiplicative structure of the field of order $q$.
We show that in many cases a direct construction of an H$(m,n)$ over $\mathbb{F}_q$, say $A$, 
can be obtained very easily by imposing that $A$ has rank 1 and, possibly, a rich group of {\it multipliers},
that are elements $u$ of $\mathbb{F}_q$ such that $u A=A$ up to a permutation of rows and columns. 
An H$(m,n)$ over $\mathbb{F}_q$ will be said {\it optimal} if the order of its group of multipliers is the least 
common multiple of the odd parts of $m$ and $n$, since this is the maximum possible
order for it. The main result is an explicit construction of a rank-one H$(m,n)$ --
reaching almost always the optimality --
for all admissible pairs $(m,n)$ for which there exist two distinct
odd primes $p$, $p'$ dividing $m$ and $n$, respectively.}

\section{Introduction}
Heffter arrays are interesting combinatorial designs introduced by Archdeacon \cite{A} in 2015. 
For a rich survey on the many variations and the related results we refer to \cite{A&J}.
In this note we consider Heffter arrays in the original meaning of ``tight" Heffter arrays extending the definition
to any group of odd order as proposed in \cite{CPP}.

 A {\it half-set} of an additive group $G$ of odd order $2\ell+1$ is a $\ell$-subset $L$ of $G$ such that $L \ \cup \ -L=G\setminus\{0\}$, i.e.,
 a complete set of representatives for the so-called {\it patterned starter} of $G$.
A subset or multisubset $S$ of an additive abelian group $G$ will be said {\it zero-sum} if the sum of all elements of $S$ is zero. This can be extended to
non-abelian groups saying that an \underline{ordered} subset or multisubset $S=\{s_1,\dots, s_k\}$ of $G$ is zero-sum if $s_1+...+s_k=0$.
 \begin{definition}
 A Heffter array H$(m,n)$ over an additive group $G$ of order $2mn+1$ is a $m\times n$ matrix whose entries form a half-set of 
 $G$ and whose rows and columns are all zero-sum.
 \end{definition} 
{\it Additive designs} have been introduced in \cite{CFP}
as {\it block designs} whose point set $V$ is a subset of an abelian group $G$ and whose blocks are all zero-sum. 
They are {\it strictly} $G$-additive when $V$ is the whole $G$.
The concrete construction of a strictly $G$-additive Steiner 2-design in which $G$ is not elementary abelian 
is one of the many hard problems concerning this interesting theory \cite{BN}.
The theory can be extended to other combinatorial designs. For instance, a Heffter array H$(n,n)$ can be seen as an additive {\it resolvable}
$(n^2_2,(2n)_n)$-{\it configuration} (see \cite{BP}).
 
A Heffter array is {\it cyclic} or {\it elementary abelian} if it is over a cyclic or elementary abelian group, respectively.
If an H$(m,n)$ is elementary abelian, then $q=2mn+1$ is obviously a prime power and we say that the array is over
$\mathbb{F}_q$, that is the field of order $q$. 

The existence problem for a cyclic H$(m,n)$ has been completely solved \cite{ABD}. On the other hand the solution
has been obtained via recursive methods and, as far as we are aware, just a few direct constructions are known.
In this note we will give direct constructions for infinitely many elementary abelian H$(m,n)$. Thus, in particular,
for infinitely many prime H$(m,n)$.

We get our constructions by imposing, first of all, that the arrays are {\it rank-one}, i.e., of rank equal to 1 in the usual sense of linear algebra.
Thus, a Heffter array H$(m,n)$ over $\mathbb{F}_q$ is rank-one if its rows are all multiples of a 
non-zero vector of $\mathbb{F}_q^n$ (and consequently its columns are all multiples of a non-zero vector 
of $\mathbb{F}_q^m$). 

\begin{example}\label{19}
It is readily seen that the following is a rank-one H$(3,3)$ over $\mathbb{Z}_{19}$.
$$A=\begin{pmatrix}
1 & 3 & -4\cr 
7 & 2 & -9\cr 
-8 & -5 & -6\cr
\end{pmatrix}$$
\end{example}

\begin{example}\label{25}
Let $g$ be a root of the primitive polynomial $x^2+x+2$ over $\mathbb{F}_{5}$. 
Then the following is a rank-one H$(3,4)$ over $\mathbb{F}_{25}$.
$$A=\begin{pmatrix}
1 && g && g+4 && 3g\cr
3g+1 && 3g+4 && 3 && 4g+2\cr
2g+3 && g+1 && 4g+3 && 3g+3
\end{pmatrix}$$
\end{example}

In \cite{CDY} Cavenagh et al. propose a definition of {\it equivalent} Heffter arrays. 
On the other hand, as far as we are aware, there is no ``official" definition of {\it isomorphic} Heffter arrays. 
From the perspective of the classic design theory, we think it is natural to give the following.
\begin{definition}
An {\it isomorphism} between a Heffter array $A$ over a group $G$ and a Heffter array $A'$ over a group $G'$
is a group isomorphism $\phi$ between $G$ and $G'$ such that $\phi(A)$ or its transposal 
can be obtained from $A'$ by suitable permutations of its rows and columns.

Thus, in particular, an {\it automorphism} of a Heffter array $A$ over a group $G$ is an automorphism of the group $G$
mapping $A$ or $A^T$ into a matrix obtainable from $A$ itself by permuting its rows and columns.

Two Heffter arrays are isomorphic if there exists an isomorphism between them.
\end{definition}
By saying that a Heffter array $A$ is over a ring $R$ we will mean that it is over the additive group of $R$.
If $R$ is with identity and $u$ is a unit of $R$, it is clear that $A$ and $uA$ are isomorphic
since the map $\widehat u: x \in R \longrightarrow ux\in R$ is an isomorphism between them.
On the other hand it is not said that $\widehat u$ is an automorphism of $A$. In the event this happens we will
say that $u$ is a multiplier of $A$.

\begin{definition}
Let $A$ be a Heffter array over a ring $R$ with identity. 
A {\it multiplier} of $A$ is any unit $u$ of $R$ such that 
$uA$ or $uA^T$ can be obtained from $A$ by suitable permutations of its rows and columns.
\end{definition}

It is evident that if $u$ is a multiplier of a Heffter array $A$, then the set of entries of $A$ and
$uA$ coincide. Thus $-1$ is never a multiplier since, by the definition of a half-set, the set of entries 
of $-A$ is exactly the complement of the set of entries of $A$. Note, however, that $A$ and $-A$ are 
equivalent in the sense of Cavenagh et al. \cite{CDY}.

The set $M$ of all multipliers of a Heffter array over a ring $R$ with identity clearly form a subgroup of the group $U(R)$
of units of $R$.
For instance, it is easy to see that the group of multipliers of the H$(3,3)$ of Example \ref{19} is $\{1,7,11\}$ 
and that the group of multipliers of the H$(3,4)$ of Example \ref{25} is $\{1,3g+1,2g+3\}$.

\medskip
Starting from the fundamental paper by R.M. Wilson \cite{W} -- or even from the earlier 
paper by R.C. Bose \cite{Bose} -- arriving to very recent 
work by the present author et al. \cite{BBGRT} the search for {\it difference families} and all their variants (a well-know topic of design theory) 
has been made easier by imposing ``many" multipliers. The same can be applied to rank-one Heffter arrays.
Indeed all the rank-one Heffter arrays constructed in this paper have a non-trivial group of multipliers.

Throughout the paper, saying that a pair $(m,n)$ is {\it admissible} we mean that $2mn+1$ is a prime power and that
both $m$ and $n$ are greater than 2. Also, speaking of ``a rank-one Heffter array H$(m,n)$" it will be understood that
$(m,n)$ is admissible and that the array is over $\mathbb{F}_{2mn+1}$.

The paper will be organized as follows. 

In the next section we will give some very elementary prerequisites that are necessary to understand the paper. 

In Section 3 it is proved that the existence of a rank-one H$(m,n)$
is completely equivalent to a factorization of a half-set of $\mathbb{F}_{2mn+1}$ into the product $X\cdot Y$ of a zero-sum 
$m$-set $X$ by a zero-sum $n$-set $Y$. As a corollary we get the easiest and nicest construction of a
rank-one H$(m,n)$: if $(m,n)$ is admissible with $m$, $n$ odd and coprime, then the $m\times n$ array $[r^{2ni+2mj}]$ where $r$ 
is a primitive element of $\mathbb{F}_{2mn+1}$, is a rank-one H$(m,n)$ whose multipliers are the non-zero squares of $\mathbb{F}_{2mn+1}$.

In Section 4 we prove that the order $\mu$ of the group of multipliers of a rank-one H$(m,n)$
is at most equal to the least common multiple of the odd parts of $m$ and $n$.
So we say that a rank-one H$(m,n)$ is {\it optimal} when $\mu$ reaches this value. 
In particular, we say that it is {\it perfect} when $\mu=mn$. The perfect rank-one Heffter arrays are
precisely the ``easiest and nicest" obtained in Section 2 and we prove that they have the useful property 
of being {\it globally simple} in Section 5.

In Section 6, using cyclotomy, we give our main result that is an explicit construction of a rank-one
H$(m,n)$ for every admissible pair $(m,n)$ such that there exist two distinct
odd primes $p$, $p'$ dividing $m$ and $n$, respectively. We also prove that this construction reaches optimality
unless $m$ and $n$ have the same radical $r$ and either $mr\over n$ or $nr\over m$ is an integer.

\section{Preliminaries}
We will need some elementary facts about {\it cyclotomy} and the related standard notation/terminology.
Given a prime power $q=de+1$, the multiplicative group of $\mathbb{F}_q$ will be denoted by $\mathbb{F}_q^*$
and the subgroup of $\mathbb{F}_q^*$ of index $e$ -- or, equivalently, the subgroup of $\mathbb{F}_q^*$
of order $d$ -- will be denoted by $C^e$. If $r$ is a primitive element of $\mathbb{F}_q^*$, then the set of cosets of $C^e$ in $\mathbb{F}_q^*$ 
(the so-called {\it cyclotomic classes of index $e$}) is $\{r^iC^e \ | \ 0\leq i\leq e-1\}$. 
It will be always understood that $r$ is fixed and, as it is standard, the coset $r^iC^e$ 
will be denoted by $C^e_i$ whichever is $i$. 
Note that we have $C^e_i=C^e_j$ if and only if $i\equiv j$ (mod $e$) and that $C^{e}_i\cdot C^e_j=C^e_{i+j}$.
Anyone who has a little bit of familiarity with finite fields should know the elementary facts below that
we recall for convenience.

\begin{lemma}\label{lemma}
Let $q$ be a prime power. Then we have:
\begin{itemize}
\item[(i)]\quad Every union of cosets of a non-trivial subgroup of $\mathbb{F}_q^*$ is zero-sum.
\item[(ii)] \quad If $q=2ed+1$ with $d$ odd, then $-1\in C^{2e}_e$ so that 
$C^{2e}_0 \ \cup \ C^{2e}_1 \ \cup \ \dots \ \cup \ C^{2e}_{e-1}$ is a half-set of $\mathbb{F}_q$.
\item[$(ii)'$] \quad
In particular, if $q=2d+1$ with $d$ odd, then $-1\in C^2_1$ and $C^2$ is a half-set of $\mathbb{F}_q$. 
\item[(iii)]\quad The product of two subgroups of $\mathbb{F}_q^*$ of orders $s$ and $t$, is the subgroup of $\mathbb{F}_q^*$ of order $lcm(s,t)$.
\end{itemize}
\end{lemma}

We also need the following.

\begin{proposition}\label{stab}
Let $q=de+1$ be a prime power and let $X$ be a subset of $\mathbb{F}_q^*$.
Then the stabilizer of $X$ under the natural action of $\mathbb{F}_q^*$ is divisible by $d$
if and only if $X$ is a union of cosets of $C^e$.
\end{proposition}
\begin{proof}
The ``if part" is obvious. Let us prove the ``only if" part.
Let $S$ be the $\mathbb{F}_q^*$-stabilizer of $X$ and assume that its order is divisible by $d$. Then $S$ contains 
the subgroup of $\mathbb{F}_q^*$ of order $d$, that is $C^e$. Thus it makes sense to consider the action
of $C^e$ on $X$. The orbits of this action partition $X$ and each of them is of the form $xC^e$ for
some $x\in X$, that is a coset of $C^e$ in $\mathbb{F}_q^*$.
The assertion follows.
\end{proof}

\section{Characterization}

One says that a subset $S$ of a multiplicative (resp.  additive) group $G$, can be {\it factorized} into the product $X\cdot Y$ (resp. sum $X+Y$)
of two subsets $X$, $Y$ of $S$ if every element of $S$ can be written in exactly one way as $x\cdot y$ (resp. $x+y$) with
$x\in X$ and $y\in Y$. One also says that $X\cdot Y$ (resp. $X+Y$) is a {\it factorization} of $S$.
By saying that $X$ is a factor of $S$ one means that $X$ is a subset of $S$ for which there exists another subset $Y$ of $S$
such that $S=X\cdot Y$ (resp. $S=X+Y$). 
In other words, $X$ is a factor of $S$ if it is possible to ``tile" $S$ with suitable translates of $X$.

We are going to show that the rank-one Heffter arrays can be characterized in terms of factorizations of a half-set of 
a finite field.

\begin{proposition}\label{prop}
Let $q=2mn+1$ be a prime power.
There exists a rank-one H$(m,n)$ if and only if there is
a suitable half-set of $\mathbb{F}_q$ admitting a factorization into the product of two
zero-sum subsets of $\mathbb{F}_q^*$ of sizes $m$ and $n$.
\end{proposition}
\begin{proof} $(\Longrightarrow).$ \quad If $A$ is a rank-one H$(m,n)$ over $\mathbb{F}_q$, then 
all columns of $A$ are multiples of a vector $(x_1,\dots,x_m)\in\mathbb{F}_q^m$. Thus, if 
$A^j$ is the $j$-th column of $A$, there is a suitable $y_j\in\mathbb{F}_q$ such that $A^j=(x_1,\dots,x_m)\cdot y_j$ for $1\leq j\leq n$.
It follows that $a_{i,j}=x_iy_j$. Hence the set of all entries of $A$, which is a  half-set of $(\mathbb{F}_q,+)$
by definition of a Heffter array, coincides with the product of the two sets $X=\{x_1,\dots,x_m\}$ and $Y=\{y_1,\dots,y_n\}$.

$(\Longleftarrow).$ \quad Let $V$ be a half-set of $\mathbb{F}_q$ and assume that 
$V=X\cdot Y$ where $X=\{x_1,\dots,x_m\}$ and $Y=\{y_1,\dots,y_n\}$ are zero-sum subsets 
of $\mathbb{F}_q^*$ of sizes $m$ and $n$, respectively. Consider the $m\times n$ array $A$ over $\mathbb{F}_q$
defined by $a_{i,j}=x_iy_j$ for $1\leq i\leq m$ and for $1\leq j\leq n$. It is obvious that the entries
of $A$ are precisely the elements of $V$, hence they form a half-set of $(\mathbb{F}_q,+)$.
Note that the $i$-th row of $A$ is $x_i\cdot (y_1,\dots,y_n)$ and that the
$j$-th column of $A$ is $(x_1,\dots,x_m)\cdot y_j$. Thus, given that both $X$ and $Y$ are zero-sum,
we infer that all rows and columns of $A$ are also zero-sum.
We conclude that $A$ is a rank-one H$(m,n)$.
\end{proof}

We will briefly refer to the sets $X$ and $Y$ in the above proof as the
{\it factors of the Heffter array $A$}.

The next result is just a special case of the more general Theorem \ref{agreeable} that we will see later.
On the other hand we think it is appropriate to give its statement now and separately, since it is an almost immediate 
consequence of Proposition \ref{prop}. Actually, this is the easiest and nicest construction of
a rank-one Heffter array.

\begin{corollary}\label{cor}\label{perfect}
If $(m,n)$ is an admissible pair with $m$, $n$ odd and coprime,
then $C^{2n}$ and $C^{2m}$ are the factors of a rank-one H$(m,n)$. 
\end{corollary}
\begin{proof}
By assumption, $q=2mn+1$ is a prime power. The set $X=C^{2n}$ and $Y=C^{2m}$ are the subgroups of $\mathbb{F}_q^*$ 
of orders $m$ and $n$, respectively. They are zero-sum by
Lemma \ref{lemma}$(i)$. Also, their product is the subgroup of $\mathbb{F}_q^*$ of order $mn$ by Lemma \ref{lemma}$(iii)$, i.e., the 
group $C^2$ of non-zero squares of $\mathbb{F}_q$ which is a half-set by Lemma \ref{lemma}$(ii)'$. The assertion
then follows from Proposition \ref{prop}.
\end{proof}

The above gives, in particular, infinitely many Heffter arrays which are elementary abelian
but not prime.
For instance, we have $7^3=2\cdot9\cdot19+1$ and hence, applying Corollary \ref{cor} 
with $m=9$ and $n=19$, we get a H$(9,19)$ over $\mathbb{F}_{7^3}$.

We will see that the group of multipliers of an H$(m,n)$ obtainable via  Corollary \ref{cor} 
is $C^2$, i.e., the group of non-zero squares of $\mathbb{F}_{2mn+1}$.

\begin{example}
Let us construct a rank-one H$(3,5)$. The request makes sense since $q=2\cdot3\cdot5+1=31$ is a prime.
Note that $x=5$ is a cubic root of unity (mod 31) and that $y=2$ is a 5th root of unity (mod 31).  
Thus $X=\langle x \rangle =\{1,5,25\}$ is the subgroup of $\mathbb{F}_{31}^*$ of order 3 
and $Y=\langle y \rangle =\{1,2,4,8,16\}$  is the subgroup of $\mathbb{F}_{31}^*$ of order 5.
The desired rank-one H$(3,5)$ is therefore
$$A=\begin{pmatrix}1 & 2 & 4 & 8 & 16\cr 5 & 10 & 20 & 9 & 18 \cr 25 & 19 & 7 & 14 & 28\end{pmatrix}$$
Note that multiplying $A$ by 9 (which is a generator of the squares of $\mathbb{F}_{31}$) corresponds to 
permute its rows cyclically $A_i \rightarrow A_{i+1 \ (mod \ 3)}$ and
then to permute the columns cyclically $A^j \rightarrow A^{j+3 \ (mod \ 5)}$. This implies that the
non-zero squares of $\mathbb{F}_{31}$ form a group of multipliers of the above array. 
\end{example}

\section{Multipliers of a rank-one Heffter array}

From now on, given any integer $k$ we denote by $k_o$ the {\it odd part of $k$}, that is 
the greatest odd divisor of $k$.

\begin{proposition}\label{multipliers}
Let $A$ be a rank-one H$(m,n)$ over $\mathbb{F}_q$ with factors $X$ and $Y$.
Then the group of multipliers of $A$ is the product of the $\mathbb{F}_q^*$-stabilizers
of $X$ and $Y$. Also, its order is at most equal to $lcm(m_o,n_o)$.
\end{proposition}
\begin{proof}
Set $X=\{x_1,\dots,x_m\}$ and $Y=\{y_1,\dots,y_n\}$ so that we have $a_{i,j}=x_iy_j$ for 
$1\leq i\leq m, 1\leq j\leq n$.
We have to prove that the group $M$ of multipliers of $A$ is the product $S\cdot T$
where $S$ and $T$ are the $\mathbb{F}_q^*$-stabilizers of $X$ and $Y$, respectively.
In the following, $A_i$ and $\overline{A_i}$ will denote the $i$-th row of $A$ and
the set of its elements, respectively. Analogously, $A^j$ and $\overline{A^j}$ will denote the $j$-th column of $A$ and
the set of its elements, respectively.

If $s\in S$, 
 then $s(x_1,\dots,x_m)=(x_{\pi(1)},\dots, x_{\pi(m)})$ for a suitable permutation
$\pi$ on the set $\{1,\dots,m\}$. It easily follows that $sA$ is the matrix whose $j$-th
column is $A^{\pi(j)}$ for $1\leq j\leq m$, hence a matrix obtainable from $A$ by permuting its columns.
Thus, by definition, $s$ is a multiplier of $A$.

Analogously, if $t\in T$, then $t(y_1,\dots,y_n)=(y_{\psi(1)},\dots, y_{\psi(n)})$ for a suitable permutation
$\psi$ on the set $\{1,\dots,n\}$. It follows that $tA$ is the matrix whose $i$-th
row is $A_{\psi(i)}$ for $1\leq i\leq n$, hence a matrix obtainable from $A$ by permuting its rows
and then $t$ is a multiplier of $A$.

Thus $M$ contains both $S$ and $T$, hence it contains their product $S\cdot T$. 

Now we prove the inverse inclusion. Let $u\in M$. By definition of a multiplier, 
there is a suitable pair $(i,j)$ such that either

\noindent
(1) $u A_1$ is a permutation of $A_i$ and $u A^1$ is a permutation of $A^j$;

or

\noindent
(2) $u A_1$ is a permutation of $A^j$ and $u A^1$ is a permutation of $A_i$.

Of course (2) may happen only in the case $m=n$.

Assume that (1) holds.
It follows, in particular, that $u a_{1,1}=u x_1y_1$ must be 
the common element of $A_i$ and $A^j$, that is $a_{i,j}=x_iy_j$.
Thus we have \begin{equation}\label{m} u=x_iy_jx_1^{-1}y_1^{-1}\end{equation}

For $1\leq i\leq m$, we have $A_i=(x_iy_1,x_iy_2,\dots,x_iy_n)$, hence $\overline{A_i}=x_iY$.
Analogously, for $1\leq j\leq n$, we have $A^j=(x_1y_j,x_2y_j,\dots,x_my_j)$, hence $\overline{A^j}=Xy_j$. 
In particular, we have $\overline{A_1}=x_1Y$ and $\overline{A^1}=Xy_1$.

The fact that $u A_1$ is a permutation of $A_i$ implies that $u \overline{A_1}=\overline{A_i}$, i.e., $u x_1Y=x_iY$.
Analogously, the fact that $u A^1$ is a permutation of $A^j$ implies that
$u\overline{A^1}=\overline{A^j}$, i.e., $u Xy_1=Xy_j$. Thus, using (\ref{m}), we get $y_jy_1^{-1}Y=Y$
and $x_ix_1^{-1}X=X$. This means that $y_jy_1^{-1}$ stabilizes $Y$ and $x_ix_1^{-1}$ stabilizes $X$,
i.e., $y_jy_1^{-1}\in T$ and $x_ix_1^{-1}\in S$. 
The conclusion is that $u=(x_ix_1^{-1})(y_jy_1^{-1})\in S\cdot T$.

If $m=n$ and (2) holds, reasoning in a very similar way one finds that $u\in S \cap T$, 
hence $u$ is the unity. Indeed, if $u$ fixes both $X$ and $Y$ then we have $u x_1=x_i$ and $u y_1=y_j$ for a suitable
pair $(i,j)$. It follows that $u x_1y_1=x_iy_1=x_1y_j$ which implies $i=j=1$, hence $u=1$.

By Proposition \ref{stab}, the order of $S$ is a divisor of $|X|=m$. Remember, however, that $-1$ cannot be a multiplier 
so that $-1\notin S$ which means that $S$ has odd order.
Hence $|S|$ is a divisor of $m_o$ and then $S$ is a subgroup of the group $M_o$ of the $m_o$-th roots of unity.
Reasoning exactly in the same way, one can see that $T$ is a subgroup of the group $N_o$ 
of the $n_o$-th roots of unity. Thus $M=S\cdot T$ is a subgroup of $M_o\cdot N_o$.
We get the assertion observing that the order of $M_o\cdot N_o$ is $lcm(m_o,n_o)$
(see Lemma \ref{lemma}$(iii)$).
\end{proof}

The above suggests the following definition.

\begin{definition}
A rank-one H$(m,n)$ is {\it optimal} if the order of its group of multipliers is the least common multiple of the 
odd parts of $m$ and $n$. It is {\it perfect} if it has order $mn$, i.e., it is the group $C^2$ of non-zero squares of $\mathbb{F}_{2mn+1}$.
\end{definition}

The perfect rank-one H$(m,n)$ can be completely characterized in view of Corollary \ref{cor}.

\begin{proposition}\label{perfect2}
There exists a perfect rank-one H$(m,n)$ if and only if $(m,n)$ is an admissible
pair with $m$, $n$ odd and coprime. 
\end{proposition}
\begin{proof}
Let $A$ be a perfect rank-one H$(m,n)$ and let $M$ be its group of multipliers. 
By definition, we have $|M|=mn$ but we also have $|M|\leq lcm(m_o,n_o)$
by Proposition \ref{multipliers}.
Thus, considering that $lcm(m_o,n_o)$ is a divisor of $mn$, we necessarily have $lcm(m_o,n_o)=mn$.
Clearly, this is possible only if $m$ and $n$ are odd and coprime. 

Now let $(m,n)$ be an admissible pair with $m$, $n$ odd and coprime. 
By Corollary \ref{cor} there exists a rank-one H$(m,n)$ whose factors
are the group $X$ of $m$-th roots of unity and the group $Y$ of $n$-th roots of unity.
Their $\mathbb{F}_q^*$-stabilizers are  $X$ and $Y$, respectively (see Proposition \ref{stab}). So the group of multipliers is
$X\cdot Y$ whose order is $mn$ by Lemma \ref{lemma}$(iii)$.
\end{proof}

\section{Globally simple rank-one Heffter arrays}

The useful notion of a {\it globally simple} Heffter array over a cyclic group has been introduced in \cite{CMPP} 
and can be extended to Heffter arrays over any group in the obvious way. In particular, a tight
H$(m,n)$, say $A=(a_{i,j})$, is globally simple  if each of its rows and each of its columns has pairwise distinct partial sums,
i.e., none of the lists $\displaystyle\sigma_i=(\sum_{k=1}^ja_{i,k} \ | \ 1\leq j\leq n)$, $i=1,\dots,m$,
and $\displaystyle\tau_j=(\sum_{k=1}^ia_{k,j} \ | \ 1\leq i\leq m)$, $j=1,\dots,n$, has repeated elements.

In general, to check whether an H$(m,n)$ is globally simple could require long and tedious calculations. 
This task is much easier in the case of rank-one Heffter arrays. 
\begin{proposition}\label{gs}
Let $A=(a_{i,j})$ be a rank-one H$(m,n)$ such that the first of its rows and the first of its columns 
have pairwise distinct partial sums. Then $A$ is globally simple.
\end{proposition}
\begin{proof}
Let $\sigma_i=(\sigma_{i,1}, \dots, \sigma_{i,n})$ be the 
list of the partial sums of $A_i$. We have $A_i=\rho_iA_1$ for a suitable $\rho_i\in\mathbb{F}_q^*$ since
$A$ is rank-one. Hence we have $\sigma_i=\rho_i\sigma_1$ which implies that $\sigma_i$ does not have 
repeated elements since this is true for $\sigma_1$ by assumption. Thus the partial sums of any row of $A$ 
are pairwise distinct.
Reasoning in the same way, we can see that the partial sums of any column of $A$ are also pairwise distinct.
\end{proof}

The perfect Heffter arrays of Corollary \ref{perfect} are globally simple.

\begin{theorem}
If $q=2mn+1$ is a prime power with $m$, $n$ odd and coprime, 
then there exists a globally simple H$(m,n)$ over $\mathbb{F}_q$.
\end{theorem}
\begin{proof}
Let $x$ be a primitive $n$-th root of unity in $\mathbb{F}_q$ and let $y$ be a primitive $m$-th root of unity in $\mathbb{F}_q$. 
By Corollary \ref{perfect}, we can take $X=(1,x,x^2,...,x^{n-1})$ and $Y=(1,y,y^2,...,y^{m-1})$ 
as first row and first column of a perfect rank-one H$(m,n)$.
The $j$-th partial sum of $X$ is  $\sigma_j=\sum_{k=0}^{j-1}x^{k}={x^{j}-1\over x-1}$ for $1\leq j\leq n$.
If we have $\sigma_{j_1}=\sigma_{j_2}$ with $j_1,j_2\in\{1,\dots,n\}$, then we have $x^{j_1-j_2}=1$, 
hence $n$ divides $j_1-j_2$ since $x$ has order $n$ in $\mathbb{F}_q^*$. This is possible only
for $j_1=j_2$, hence the partial sums of $X$ are pairwise distinct. 

Analogously, the $i$-th partial sum of $Y$ is  $\tau_i=\sum_{k=0}^{i-1}y^{k}={y^{i}-1\over y-1}$ for $1\leq i\leq m$.
If we have $\tau_{i_1}=\tau_{i_2}$ with $i_1,i_2\in\{1,\dots,m\}$, then we have $y^{i_1-i_2}=1$, 
hence $m$ divides $i_1-i_2$ since $y$ has order $m$ in $\mathbb{F}_q^*$. This is possible only
for $i_1=i_2$, hence the partial sums of $Y$ are pairwise distinct. 

The assertion then follows from Proposition \ref{gs}.
\end{proof}

\section{Main results}\label{agreeablepairs}

Let us say that an admissible pair $(m,n)$ is {\it agreeable} if there exist two distinct odd 
primes $p$ and $p'$ dividing $m$ and $n$, respectively.
We are going to give an explicit construction of a rank-one H$(m,n)$ for any admissible agreeable pair $(m,n)$.
Then we will discuss its possible optimality.

\begin{theorem}\label{agreeable}
If $q=2mn+1$ is a prime power with $(m,n)$ agreeable, then there exists a rank-one H$(m,n)$ over $\mathbb{F}_q$.
\end{theorem}
\begin{proof}
Take any pair $(m_1,n_1)$ of odd coprime integers greater than 1 with $m_1$ dividing $m_o$ and $n_1$ dividing $n_o$.
Such a pair certainly exists by definition of an agreeable pair; in the worst of the cases $m_1$ and $n_1$
may be suitable odd primes.
Set $m=m_1m_2$, $n=n_1n_2$ and consider the following subsets of $\mathbb{F}_q^*$:
$$X=\bigcup_{i=0}^{m_2-1}C^{2m_2n}_i;\quad\quad Y=\bigcup_{j=0}^{n_2-1}C^{2mn_2}_{jn_2}.$$
We have $2m_2n={q-1\over m_1}$, hence $C^{2m_2n}$ is the subgroup of $\mathbb{F}_q^*$ of order $m_1$.
So we see that $X$ is a union of $m_2$ cosets of this subgroup, hence a $m$-subset of $\mathbb{F}_q^*$
which is zero-sum by Lemma \ref{lemma}$(i)$.
Analogously, we have $2mn_2={q-1\over n_1}$, hence $C^{2mn_2}$ is the subgroup of $\mathbb{F}_q^*$ of order $n_1$.
Thus $Y$ is a union of $n_2$ cosets of this subgroup, hence a $n$-subset of $\mathbb{F}_q^*$
which is zero-sum by Lemma \ref{lemma}$(i)$.

Given that $\gcd(m_1,n_1)=1$, by Lemma \ref{lemma}$(iii)$ the product of the subgroups of $\mathbb{F}_q^*$
of orders $m_1$ and $n_1$ is the subgroup of $\mathbb{F}_q^*$ of order $m_1n_1$:
$$C^{2m_2n}\cdot C^{2mn_2}=C^{2m_2n_2}$$ 
It follows that 
$$C^{2m_2n}_i\cdot C^{2mn_2}_j=C^{2m_2n_2}_{i+j}$$ 
and then we can write:
$$X\cdot Y=\displaystyle\bigcup_{(i,j)\in I\times J}C^{2m_2n_2}_{i+j}$$
with $I=\{0,1,\dots,m_2-1\}$ and $J=\{0,m_2,2m_2,\dots,(n_2-1)m_2\}$.
Now note that $I+J$ is a factorization of the whole interval $[0,m_2n_2-1]$. Thus, setting $e=m_2n_2$, we can write
$$X\cdot Y=C^{2e}_0 \ \cup \ C^{2e}_1 \ \cup \ \dots \ \cup \ C^{2e}_{e-1}$$ 
It follows, by Lemma \ref{lemma}$(ii)$, that
$X\cdot Y$ is a half-set of $\mathbb{F}_q$.
Then, if we take an arbitrary ordering $(x_1,\dots,x_m)$ of $X$ and
an arbitrary ordering $(y_1,\dots,y_n)$ of $Y$, using the ``if part" of Proposition \ref{prop}
we can finally say that the $m\times n$ array $[x_iy_j]$ is a rank-one H$(m,n)$.
\end{proof}

\begin{example}
Let us construct a rank-one H$(6,15)$. The pair $(6,15)$ is agreeable. Indeed $q=2\cdot6\cdot15+1=181$ is a prime and
$m_1=3$, $n_1=5$ are distinct primes dividing the odd parts of 6 and 15, respectively. Using $r=2$ as primitive element of $\mathbb{F}_{181}$ and 
following the instructions of Theorem \ref{agreeable} we find that
the factors of a rank-one H$(6,15)$ are 
\small
$$X=C^{60}_0 \ \cup \ C^{60}_1=\{1, 48, 132, 2, 96, 83\} \quad {\rm and}$$
$$Y=C^{36}_0 \ \cup \ C^{36}_2 \ \cup \ C^{36}_4=\{1, 59, 42, 125, 135, 4, 55, 168, 138, 178, 16, 39, 129, 9, 169\}.$$
\normalsize
Therefore the desired H$(6,15)$ is the following
\scriptsize
$$\left(\begin{array}{*{24}c}
1& 59& 42& 125& 135& 4& 55& 168& 138& 178& 16& 39& 129& 9& 169 \cr
48& 117& 25& 27& 145& 11& 106& 100& 108& 37& 44& 62& 38& 70& 148 \cr
132& 5& 114& 29& 82& 166& 20& 94& 116& 147& 121& 80& 14& 102& 45 \cr
2& 118& 84& 69& 89& 8& 110& 155& 95& 175& 32& 78& 77& 18& 157 \cr
96& 53& 50& 54& 109& 22& 31& 19& 35& 74& 88& 124& 76& 140& 115 \cr
83& 10& 47& 58& 164& 151& 40& 7& 51& 113& 61& 160& 28& 23& 90
\end{array}\right)$$
\end{example}

We recall  that the 
{\it radical} of an integer $k$, denoted by $rad(k)$, is the product of the individual prime factors of $k$ if $k>1$, and it is equal to 1
if $k=1$. We give the following definition.
\begin{definition}
An admissible pair $(m,n)$ is {\it optimal} if it is agreeable and neither $m_o\cdot rad(m_o)$ divides $n_o$ nor $n_o\cdot rad(n_o)$ divides $m_o$.
\end{definition}

We want to exploit the proof of Theorem \ref{multipliers} in the best possible way in order to establish
whether it may lead to an optimal rank-one Heffter array.

\begin{theorem}\label{optimality}
If an admissible pair $(m,n)$ is optimal, then there exists an optimal rank-one H$(m,n)$.
\end{theorem}
\begin{proof}
Let $A$ be the rank-one H$(m,n)$ constructed in Theorem \ref{agreeable},
let $X$, $Y$ be its factors and let $S$, $T$ be their respective $\mathbb{F}_q^*$-stabilizers.
Recall that $X$ is a union of cosets of the $m_1$-th roots of unity and that $Y$ is a union of cosets of the $n_1$-th roots 
of unity. Then, by Proposition \ref{stab}, the orders of $S$ and $T$ are
divisible by $m_1$ and $n_1$, respectively.
Recalling that the group of multipliers of $A$ is $S\cdot T$ by Proposition \ref{multipliers}, 
we deduce that its order is at least equal to $m_1n_1$. 

Thus, to prove the assertion it is enough to show that $lcm(m_o,n_o)$ can be written as a product 
$m_1n_1$ with $m_1>1$ a divisor of $m_0$, $n_1>1$ a divisor of $n_0$, and $\gcd(m_1,n_1)=1$.

Consider first the case $m_o=n_o$. Here, to say that $(m,n)$ is optimal simply means that
$m_o$ has at least two distinct prime factors. 
Let $p$ be one of them and let $p^\alpha$ be the largest
power of $p$ dividing $m_o$. Then $m_1=p^\alpha$ and $n_1={m_o\over p^\alpha}$ satisfy the 
requirement.

Now assume that $m_o<n_o$. 
Let $\{p_1,\dots,p_t\}$ be the set of prime divisors of $m_on_o$,
let $p_i^{\alpha_i}$ be the largest power of $p_i$ dividing $m_o$, and let
$p_i^{\beta_i}$ be the largest power of $p_i$ dividing $n_o$.
Now let $I$ be the set of $i$'s such that $\alpha_i\geq \beta_i$ and let $J$ be the complement of $I$
in $\{1,\dots,t\}$, i.e., the set of $j$'s such that $\alpha_j<\beta_j$.
It is obvious that $lcm(m_o,n_o)=\prod_{i=1}^t p_i^{\max(\alpha_i,\beta_i)}$, hence
we can write $lcm(m_o,n_o)$ as a product of a divisor $m_1$ of $m_o$ and a divisor $n_1$ of $n_o$ as follows
$$lcm(m,n)=m_1n_1 \quad{\rm with}\quad m_1=\prod_{i\in I}p_i^{\alpha_i} \quad{\rm and}\quad n_1=\prod_{j\in J}p_j^{\beta_j}.$$
It is also obvious that $m_1$ and $n_1$ are coprime. It remains to show that $m_1$ and $n_1$ are both greater than 1. 
Indeed, if $m_1=1$, then $I$ would be empty
which means that, for every $h$, the exponent of $p_h$ in the prime factorization of $m_o$ is 
strictly less than the exponent of $p_h$ in the prime factorization of $n_o$. 
This clearly implies that $m_o\cdot rad(m_o)$ is a divisor of $n_o$
contradicting that $(m,n)$ is optimal. 
Also, if $n_1=1$ then $J$ would be empty which means that, for every $h$, the  exponent of $p_h$ in the prime factorization of $m_o$ is
greater than or equal to the exponent 
of $p_h$ in the prime factorization of $n_o$. This would imply that $m_o\geq n_o$ against the assumption. 
We conclude that $m_1$, $n_1$ satisfy the requirement.

The remaining case $m_o>n_o$ can be proved exactly in the same way by inverting the roles of $m_o$ and $n_o$.
\end{proof}

Note that the agreeable pairs which are not optimal are, up to the order, of the form $(2^\alpha\nu\rho,2^\beta\nu)$
where $\nu$, $\rho$ are odd integers such that $\nu$ has at least two distinct prime divisors and $rad(\rho)=rad(\nu)$.
The least admissible pair which is agreeable but not optimal is $(21^2,21)$. Indeed the pair $(15^2,15)$ is not admissible
since we have $2\cdot15^3+1=43\cdot157$.

\section{Conclusion}

We have constructed, explicitly, a rank-one H$(m,n)$ for any admissible agreeable pair $(m,n)$. 

The existence question for rank-one Heffter arrays remains open for admissible {\it disagreeable} pairs, which are
precisely the pairs $(m,n)$ where either one of the two parameters is equal to a power of 2, or their product has exactly one odd prime factor. 
Apart from some examples where $m$ and $n$ are ``small" (see, e.g,. Example \ref{25}), 
we doubt that rank-one ``disagreeable" Heffter arrays are obtainable by means of a very explicit construction. 
Yet, we strongly believe that their existence can be proved in any admissible case by means of some tools
which are a little bit more sophisticated than those used in this paper. We hope to find the final proof in a future work. 

\normalsize
\section*{Acknowledgement}
The author would like to thank Anita Pasotti for helpful discussions on the topic and
the anonymous referee who pointed out that the perfect rank-one Heffter arrays
of Corollary \ref{perfect} are globally simple.

This work has been performed under the auspices of the G.N.S.A.G.A. of the
C.N.R. (National Research Council) of Italy.

\end{document}